\documentclass[12pt]{article}
\usepackage{amssymb}
\usepackage{enumerate}
\usepackage{amsfonts}
\usepackage{graphicx}
\usepackage{amsmath}
\usepackage{amsthm}
\usepackage{amssymb}
\usepackage[colorlinks]{hyperref}  
\usepackage{color}
\usepackage{fullpage}

    \newtheorem{theorem}{Theorem}
    \newtheorem{lemma}[theorem]{Lemma}
    \newtheorem{proposition}[theorem]{Proposition}
    \newtheorem{corollary}[theorem]{Corollary}

\theoremstyle{definition} % For roman text in the body
    
    \newtheorem{remark}[theorem]{Remark}

\newcommand{\fip}[1]{{\prec \hspace{-0.25em} #1\hspace{-0.25em}\succ}}
\newcommand{\eps}{\varepsilon}

\newcommand{\N}{{\mathcal N}}

\newcommand{\lstar}{{\raise-0.15ex\hbox{$\scriptstyle \ast$}}}

\newcommand{\var}{\text{Var }}
\theoremstyle{remark} % For an italic header, more subtle than definition style

\newcommand{\rr}{\mathbb{R}}

\newcommand{\ed}{\stackrel{d}{=}}

\newcommand{\airy}{\textup{Airy}_\beta}

\newcommand{\G}[2]{\mathcal{G}_{#2}^{(#1)}}
\newcommand{\Hop}[1]{\mathcal{H}_{\beta}^{(#1)}}
\newcommand{\Hmat}[1]{H_{n,\beta}^{(#1)}}
\newcommand{\varp}[2]{\varphi_{#1,s,t}^{(#2)}}
\newcommand{\psip}[2]{\psi_{#1,s,t}^{(#2)}}
\newcommand{\Span}{\textup{Span}\,}

\begin{document}

\title{Evolution of the stochastic Airy eigenvalues under a changing boundary}

\author{Angelica Gonzalez\footnote{University of Arizona, agonzalez@math.arizona.edu}\hspace{.8cm} Diane Holcomb\footnote{KTH, holcomb@kth.se}}

\maketitle

\begin{abstract}
The $\airy$ point process, originally introduced by Ram\'irez, Rider, and Vir\'ag \cite{RRV}, is defined as the spectrum of the stochastic Airy operator $\mathcal{H}_\beta$ acting on a subspace of $L^2[0,\infty)$ with Dirichlet boundary condition. In this paper we study the coupled family of point processes defined as the eigenvalues of $\mathcal{H}_\beta$ acting on a subspace of $L^2[t,\infty)$. These point processes are coupled through the Brownian term of $\mathcal{H}_\beta$. We show that these point processes as a function of $t$ are differentiable with explicitly computable derivative. Moreover when recentered by $t$ the resulting point process is stationary. This process can also be viewed as an analogue to the `GUE minor process' in the tridiagonal setting.
%This process was first observed by Bal\'int Vir\'ag.
%We derive these properties and give a characterization of the inverse process by making use of the Ricatti equations associated to the operator. 
\end{abstract}

\section{Introduction}

In this paper we work with a generalization of Gaussian Orthogonal, Unitary, and Symplectic ensembles, which were first introduced by Wigner in the 50's. These matrix models have many unique properties including an explicitly computable eigenvalue distribution given by
\begin{align}
\label{eq:hermite}
p_{H}(\lambda_1,\dots,\lambda_n)=\frac{1}{Z_{n,\beta}} \prod_{1\le i<j\le n} |\lambda_i-\lambda_j|^\beta   \prod\limits_{i=1}^ne^{-\frac{\beta}4  \lambda_i^2},
\end{align}
for $\beta = 1, 2, $ or $4$. The $\beta$-Hermite ensemble generalizes this to a set of $n$ points on the line whose joint density is given by \eqref{eq:hermite} for any $\beta>0$.  This point process is no longer related to a full matrix model, but it does have an associated tridiagonal matrix model. The model originally introduced by Dumitriu and Edelman \cite{DE} is as follows:  Let
\begin{align}
\label{eq:tridiagonal}
A_\beta \sim \frac{1}{\sqrt \beta} \left[ \begin{array}{ccccc} \N(0,2) & \chi_{(n-1)\beta} &&& \\ \chi_{(n-1)\beta} & \N(0,2) & \chi_{(n-2)\beta} && \\ & \ddots & \ddots & \ddots & \\ && \chi_{2\beta} & \N(0,2) & \chi_\beta \\ &&& \chi_\beta & \N(0,2) \end{array} \right],
\end{align}
with all of the entries independent. The $\chi$ random variables are subscripted by their parameter. In the case where $k$ is an integer, a $\chi_k$ random variable has the same distribution as the norm of a vector in $\rr^k$ with independent $\N (0,1)$ entries. These is a natural generalization in the non-integer case $k>0$.  

Edelman and Sutton observed that this matrix model may be seen as an operator on step functions, and using this observation conjectured that in the limit the upper edge of the spectrum will converge to a certain differential operator \cite{ES}. Indeed, in this setting at the upper and lower edge of the spectrum Ram\'irez, Rider, and Vir\'ag showed that the centered and scaled matrix model converges in a weak sense to the  ``stochastic Airy operator" (denoted here by $\mathcal{H}_\beta$) which in turn is used to show convergence of the eigenvalues \cite{RRV}.  Let 
\begin{equation}
\mathcal{H}_\beta= - \frac{d^2}{dx^2}+x+ \frac{2}{\sqrt \beta} b'(x)
\end{equation}
where we take $b'$ to be a white noise.  A precise definition and many properties of this operator  can be found in \cite{RRV}.  We review the necessary ones below. 

For our purposes it is sufficient to define an eigenfunction/eigenvalue pair in the following way:  Let 
\[L^*[t,\infty)= \left\{f\in L^2[t,\infty)|\ f(t)=0, f' \text{ exists a.e. and } \int_t^\infty (f')^2+(1+x)f^2 dx<\infty\right\},\]
then $(\varphi ,\lambda)$ is an eigenvalue/eigenfunction pair for $\mathcal{H}_\beta$ acting on $L^*[t, \infty)$  if $\|\varphi\|_2=1$, and 
\begin{equation}
\varphi ''(x)= \frac{2}{\sqrt \beta} \varphi(x)b'(x)+(x-\lambda)\varphi(x)
\end{equation}
holds in the sense of distributions.  This may be written as 
\begin{equation}
\label{eq:eveft}
\varphi'(x)- \varphi'(t)= \frac{2}{\sqrt \beta} \varphi(x)b(x)-\frac{2}{\sqrt \beta} \int_t^x \varphi'(s)b(s)ds+ \int_s^x(s-\lambda)\varphi (s)ds.
\end{equation}
In this sense, the set of eigenvalues is a deterministic function of the Brownian path $b$.  Note that this is a slight generalization from the case considered in \cite{RRV} where they focused on $\mathcal{H}_\beta$ acting on functions in $L^*[0,\infty)$.

The eigenvalues of $\mathcal{H}_\beta$ acting on $L^*[t,\infty)$ are ``nice'' in the following sense: 
\begin{theorem}
\text{\cite{RRV}}
With probability one, the eigenvalues of $\mathcal{H}_\beta$ are distinct (of multiplicity 1) with no accumulation point, and for each $k\geq 0$ the set of eigenvalues of $\mathcal{H}_\beta$ has a well defined $(k+1)$st lowest element $\Lambda_k(\beta)$.
\end{theorem}

In this paper we study the evolution of the eigenvalues of $\mathcal{H}_\beta$ acting on $L^*[t,\infty)$ as a process in $t$. That is we consider the operator $\mathcal{H}_\beta$ acting on $L^*[t,\infty)$ and study the evolution of the eigenvalues as $t$ varies. We will denote the operator acting on the particular domain by 
\begin{equation}
\label{eq:Ht}
\mathcal{H}_\beta^{(t)}= - \frac{d^2}{dx^2}+x+ \frac{2}{\sqrt \beta}b'_x, \qquad \Hop{t}: L^*[t,\infty) \to L^2[t, \infty),
\end{equation} 
and define $\Lambda_1(t)< \Lambda_2(t)< \Lambda_3(t)< \cdots$ to be the ordered eigenvalues of $\mathcal{H}_\beta^{(t)}$. We observe that the eigenvalue/ eigenfunction condition may be written in the same way as before, but also has an interpretation in terms of a shifted Brownian motion. That is $(\varphi, \lambda)$ is an eigenvalue/eigenfunction pair of $\Hop{t}$ if for $x\ge t$
\begin{align}
\varphi'(x)- \varphi'(t)&=  \frac{2}{\sqrt \beta} \varphi(x)(b(x)-b(t))-\frac{2}{\sqrt \beta} \int_t^x \varphi'(s)(b(s)-b(t))ds+ \int_t^x(s-\lambda)\varphi (s)ds .
\end{align}

\begin{theorem}
\label{thm:process}
Let $k$ be any fixed positive integer and let $\mathcal{G}_t^{(k)} = \{\Lambda_1(t),..., \Lambda_k(t)\}$. The process $\mathcal{G}_t^{(k)}$ is differentiable in time and for every fixed $t$ we have that
\begin{equation}
\label{eq:derivativedist}
\frac{d}{dt} \mathcal{G}_t^{(k)} \ed \{\Gamma_{1}(t),..., \Gamma_{k}(t)\}, \qquad \text{ with i.i.d.}  \quad \Gamma_{i}(t) \sim \Gamma(\tfrac{\beta}{2}, \tfrac{2}{\beta}).
\end{equation}
Moreover, the process $\mathcal{G}_t^{(k)}-t$ is stationary.
\end{theorem}

\begin{remark}
Note that in the above characterization the eigenvalue/eigenfunction pairs are defined in a path--wise sense. In this paper all calculations unless otherwise noted should be understood in this sense. Because of this it is sufficient to prove various estimates and limits for an arbitrary value of $t$ and any Brownian path in a set of full measure. 
\end{remark}

\begin{corollary}
The process $\G{t}{k}-t$ is not reversible.
\end{corollary}

\begin{proof}
To see that the process cannot be reversible it is enough to observe that for a single eigenvalue the distribution of the derivative of the forward process is $\Gamma(\frac{\beta}{2}, \frac{2}{\beta}) -1$. On the other hand the distribution of the derivative of the reversed process is $1-\Gamma (\frac{\beta}{2}, \frac{2}{\beta})$, which are not equivalent.  
\end{proof}

%Observe each particle in the $\G{t}{k}$ process has the property that the derivative is positive. This implies the existence of an inverse process. This inverse process may be characterized directly via the Ricatti diffusion associated to $\mathcal{H}_\beta$. This is discussed further in the final section of the paper.  

While the study of the eigenvalues of $\mathcal{H}_\beta$ on a changing domain is itself interesting it also has a connection to the  original tridiagonal model in (\ref{eq:tridiagonal}). Moreover this connection may be used to derive properties of the limiting process including the distribution of the derivatives. We are interested in the behavior of the spectrum at the upper and so begin by centering at $2\sqrt{n}$. We denote the centered, truncated matrix obtained by removing the first $k-1$ rows and columns by
\begin{equation}
\label{eq:ktridiagonal}
\Hmat{k} = 2\sqrt{n}I -\frac{1}{\sqrt \beta} \left[ \begin{array}{ccccc} N(0,2) & \chi_{(n-k)\beta} &&& \\ \chi_{(n-k)\beta} & N(0,2) & \chi_{(n-k-1)\beta} && \\ & \ddots & \ddots & \ddots & \\ && \chi_{2\beta} & N(0,2) & \chi_\beta \\ &&& \chi_\beta & N(0,2) \end{array} \right],
\end{equation}
and denote its ordered eigenvalues by $\lambda_1(n,k) < \lambda_{2}(n,k)< \cdots < \lambda_{n-k+1}(n,k)$.

\begin{theorem}
\label{thm:discretetocont}
Suppose that $\lambda_1(n,k) < \lambda_{2}(n,k)< \cdots $ be defined as above then 
\begin{equation}
\Big( \{n^{1/6}\lambda_i(n,\lfloor n^{1/3}t\rfloor)\}_{i=1}^{k}, t\ge 0 \Big) \Rightarrow \mathcal{G}_t^{(k)}
\end{equation}
where $\G{t}{k}$ eigenvalue process of $\mathcal{H}_\beta$ defined above. 
\end{theorem}

\begin{figure}
\begin{center}
\includegraphics[height=2in]{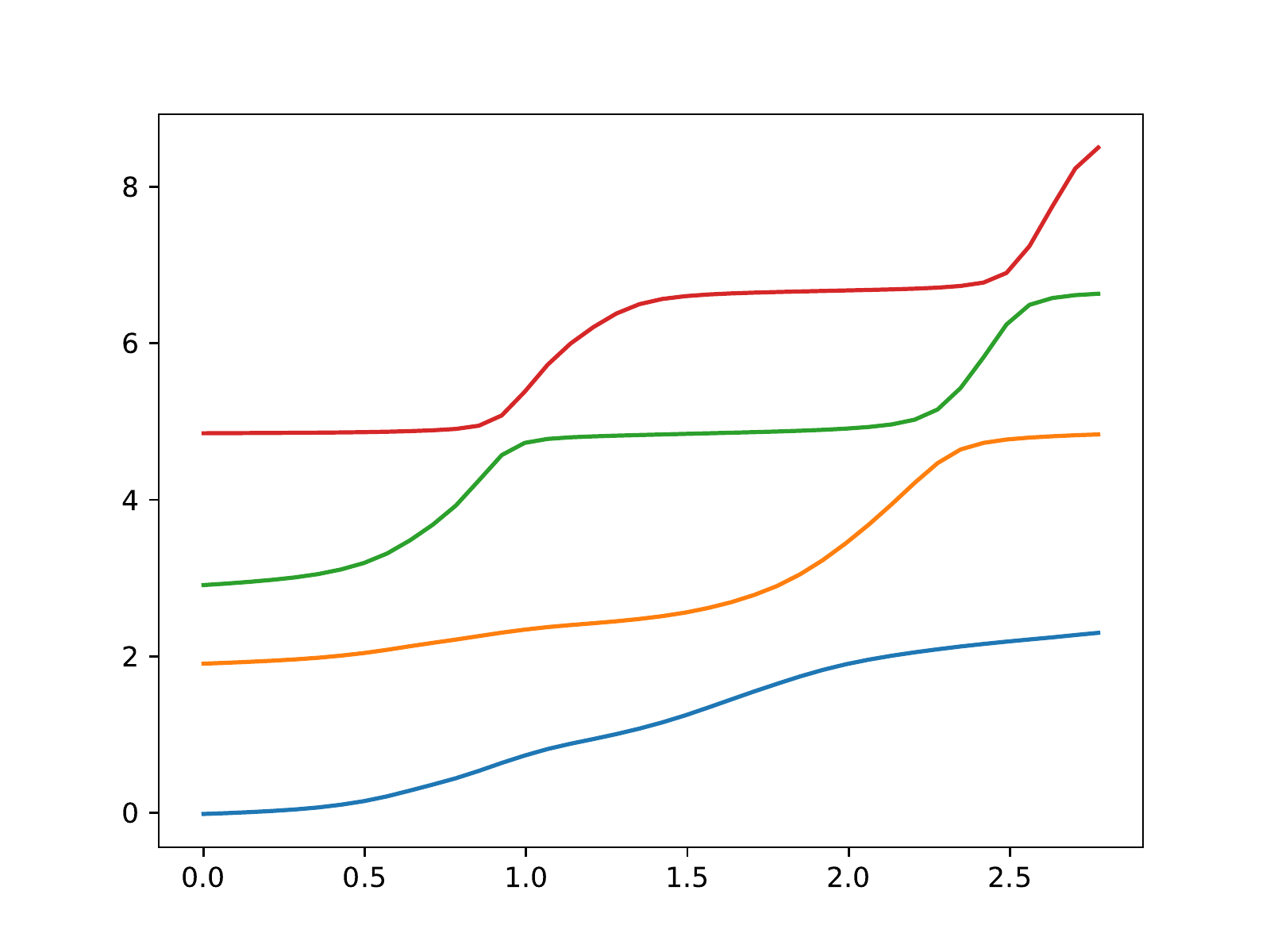}
\hspace{.5cm}
\includegraphics[height=2in]{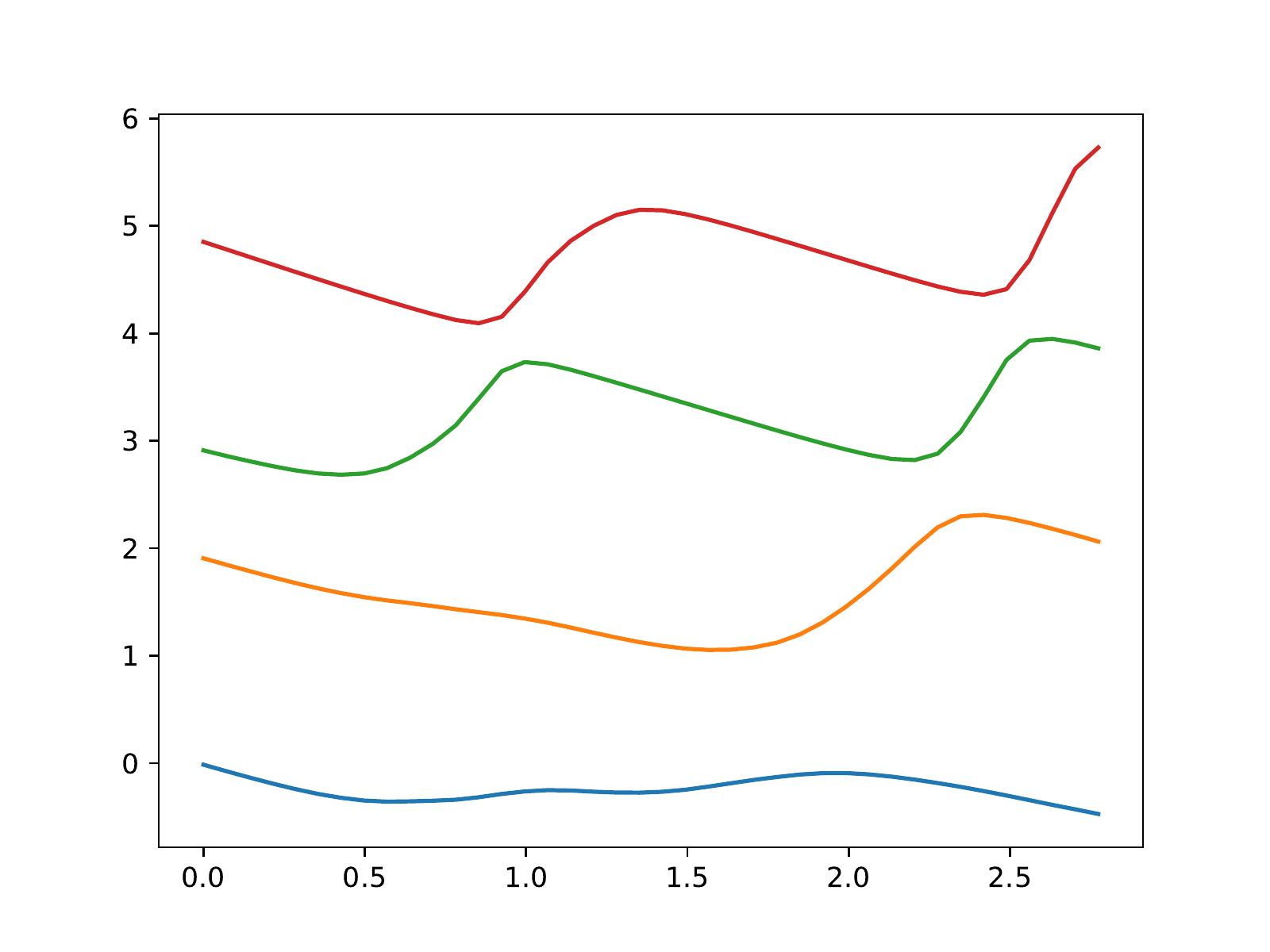}
\end{center}
\caption{Bottom eigenvalues $n^{1/6}\lambda_i(n,\lfloor n^{1/3}t\rfloor)$ of a $n^{1/6}H_{n,2}^{(\lfloor t n^{1/3}\rfloor)}$ matrix for $i=1,...,5$ as a function of $t$ with $n=6000$ before and after recentering by $t$.}
\end{figure}

The reader might notice at this point that we are essentially considering the `minor process' associated to the tridiagonal matrix model. This turns out to define a very different process than the classical `GUE minor process' when $\beta=2$ that is derived from the submatrices of the full matrix model. For more details on this classical process see \cite{FN1}. In particular the eigenvalues of that process follow rough paths. The same process may be realized by considering appropriate limits of Dyson Brownian Motions \cite{FN2}. The fact that two different process are obtained is particularly interesting in light of the fact that for both models when one considers the sub-matrix obtained by removing the first $k$ rows and columns they again have the same eigenvalue distributions, and in both cases eigenvalues of successive sub-matrices satisfy interlacing.

The paper will be organized as follows: We begin recalling properties of $\mathcal{H}_\beta$ and showing that the process $\mathcal{G}_t^{(k)}$ is stationary and differentiable. In the next section we show the convergence statement in \ref{thm:discretetocont}. Finally, in the last section we use the convergence statement to determine the distribution of the derivative vector.

\noindent \textbf{Acknowledgements:} The authors would like to thank B\'alint Vir\'ag for the problem suggestion and discussions. The work of the second author was supported in part by funding from the Knut and Alice Wallenberg foundation award number KAW 2015.0359, and Swedish Research Council award number 2018-04758.

\section{On the eigenvalues of the restricted operator}

\begin{proposition}
For any fixed $k$ the process $\G{k}{t}-t$ is stationary as a process in $t$. 
\end{proposition}

\begin{proof}
We use definition that $(\lambda, \varphi)$ is an eigenvalue/eigenfunction pair for $\Hop t$ if (\ref{eq:eveft}) is satisfied, we define the time sifted function $\psi(x-t)= \varphi(x)$ and shifted Brownian motion $w(x-t)=b(x)-b(t)$ then $\psi$ satisfies the equation
\begin{equation}
\psi'(x-t)- \psi'(0)= \frac{2}{\sqrt \beta} \psi(x-t)w(x-t)-\frac{2}{\sqrt \beta} \int_0^{x-t} \psi'(s)w(s)ds+ \int_0^{x-t}(s-(\lambda-t))\psi (s)ds.
\end{equation}
This is equivalent in distribution to $\lambda-t$ being an eigenvalue of 
\[
\mathcal{H} = - \frac{d^2}{dx^2} + x + \frac{2}{\sqrt{\beta}} w'_x \ed \mathcal{H}_\beta^{(0)}.
\] 

Therefore the lowest $k$ eigenvalues of $\Hop{t}$ shifted by $t$ have the same distribution as the lowest $k$ eigenvalue of $\mathcal{H}_\beta^{(0)}$ for all $t$ and so $\G{k}{t}-t$ is stationary.  
\end{proof}

%%%%%%%%%%%%%%%%%%%%%%%%%%%%%%%%%%%%%%%%%%%%%%%%%%%%%%%%%%%%%%%%%%%%%%%%%%%%%%%%%%%%%%%%%%%%%%%%%%%%%%%%%%%%%%%%%%%%%%%%%%%%%%%%%%%%%%%%%%%%%%%%%%%%%%%%%%%%%%%%%%%%%%%%%%%%%%%%%%%%%%%%%%%%%%%%%%%%%%%%%

We let $f_{k,t}\in L^*[t,\infty)$ denote the eigenfunction associated to $\Lambda_{k}(t)$ the $k$th lowest eigenvalue of $\Hop{t}$. The idea for the remainder of this section will be to approximate the eigenfunction $f_{k,t}$ by using eigenfunctions at $f_{k,t+\eps}$ and $f_{k,t-\eps}$ and replacing the starting section of the eigenfunction with just a straight line. We make the following definitions: 

For every pair $s<t$ we define two new families of functions
\begin{equation}
\label{eq:linearapprox}
\phi_{k,s,t}^{(a)}(x) = \begin{cases} 
(x-s) \frac{f_{k,t}(a)}{a-s} &s\le x < a\\
f_{k,t}(x) & x \ge  a
\end{cases}, 
\qquad
\psip{k}{a}(x)= \begin{cases}
(x-t)\frac{f_{k,s}(a)}{a-t} &t\le x < a \\
f_{k,s}(x) & x \ge a
\end{cases}.
\end{equation}
The function $\varp{k}{a}$ approximates the $k$th eigenfunction for $\Hop{s}$ by building a function from the $k$th eigenfunction of $\Hop{t}$. The function $\psip{k}{a}$ does something similar, but instead approximates the $k$th eigenfunction of $\Hop{t}$ by looking at the $k$th eigenfunction of $\Hop{s}$. See figure \ref{functionbuilding} for an illustration of how $\phi$ and $\psi$ are constructed from a function $f$.

\begin{figure}[h]
\label{functionbuilding}
\begin{center}
\includegraphics[height=2in]{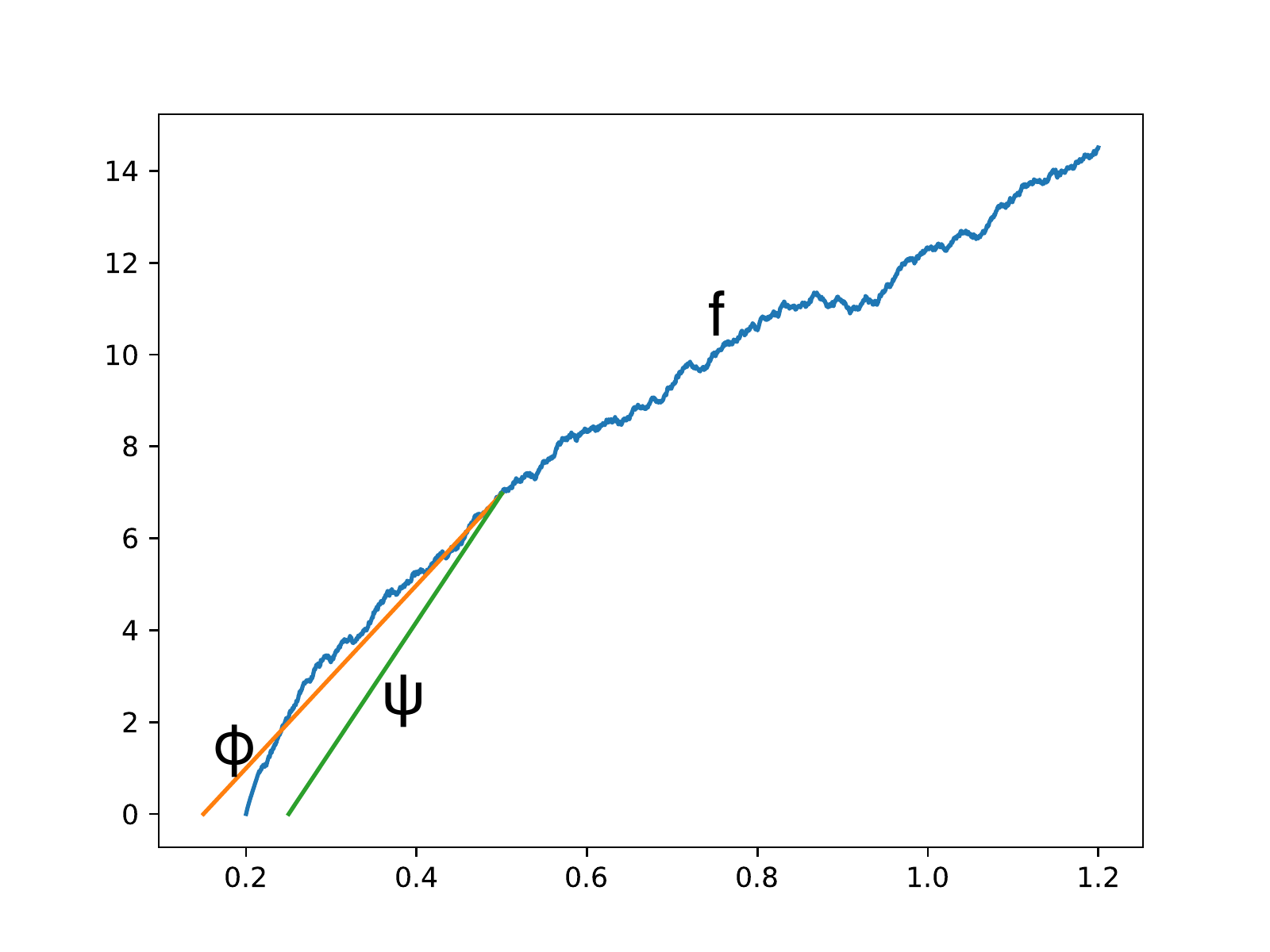}
\caption{Building $\phi$ and $\psi$ from a function $f$.}
\end{center}
\end{figure}

The idea here will be to make use of the variational characterization of the eigenvalues: 
\begin{equation}
\label{eq:characterization}
\Lambda_k(t) = \inf_{B\subset L^*_t, \dim B=k} \sup_{g\in B} \frac{\fip{ g, \Hop t g}}{\fip{ g, g }} = \fip{ f_{k,t}, \Hop t f_{k,t}}. 
\end{equation}
It follows immediately from the variational characterization that 
\begin{align*}
\Lambda_1(t) &= \fip{ f_{1,t}, \Hop t f_{1,t}} \\
& \le\  \frac{\fip{ \phi_{1,t,t+\eps}^{(a)}, \Hop t \phi_{1,t,t+\eps}^{(a)}}}{\fip{ \phi_{1,t,t+\eps}^{(a)},  \phi_{1,t,t+\eps}^{(a)}}} = \Lambda_1(t+\eps) + \text{error},
\end{align*}
with a similar bound holding using $\psi_{1,t-\eps,t}^{(a)}$ and $\Lambda_1(t-\eps)$. The remainder of this section is devoted to showing that the error is of order $\eps$ and identifying an expression for the derivative of $\Lambda_k(t)$.

The first step in showing that $\G k t$ is a differentiable process is to prove that the error is of order $\eps$. This requires several results on the eigenfunctions of the operators $\Hop t$.  The first result will be to show that the eigenfunctions are `close' to linear near their boundary, which we will need to that $\phi_{k,s,t}^{(a)}$ and $\psip{k}{a}$ are good approximations. Note that for $t>0$ we can extend $f_{k,t}$ to a function on $L^*[0,\infty)$ by taking $f_{k,t}(x)=0$ for $x<t$. This extension should be implicitly understood where necessary in the following computations.

\begin{proposition}
\label{prop:almostlinear}
	Let $\varphi_t(x)$ be an eigenfunction of $\mathcal{H}_\beta(t)$.  Suppose that $\sup_{x\in [t, x_0]}|\varphi_t'(x)- \varphi_t'(t)| < \eta_{x_0}$ for some $\eta_{x_0}>0$.  Then for every $0<\delta< 1/2$ and $x \in [t,x_0]$ there exists $C_{x_0,\delta}$, 
	\begin{align*}
	|\varphi_t(x)  - (x-t) \varphi_t'(t) | \leq C_{b,x_0,\delta}(\eta_{x_0}+\varphi_t'(t)) |x-t|^{2+\delta}.
	\end{align*}
\end{proposition}

\begin{proof}
We begin with two bounds. \\
\noindent \emph{Bound on Brownian Motion:}  Using that Brownian Motion $\alpha$-Holder continuous for $\alpha<1/2$ and for all $t>0$ we have that almost surely for all BM paths and $x,t<x_0$ 
\begin{align}
\label{eq:BMbnd}
|b(x)-b(t)| \leq C_{x_0,\delta} |x-t|^{\delta},\text{ for } 0 < \delta <1/2. 
\end{align} 
\emph{Bound for $\varphi_t(x)$:}  We apply the Mean Value Theorem to $\varphi_t(x)$ to get that for some $r \in (t,x)$
\begin{align}
\varphi_t(x) = \varphi_t'(r) \cdot (x-t) \leq (\eta_{x_0} +\varphi_t'(t)) (x-t). \label{eq:phibnd} 
\end{align}

	Now suppose that $\varphi_t(x)$ is the eigenfunction corresponding to eigenvalue $\lambda$. Then 
	\begin{align*}
	\varphi_t'(x)-\varphi_t'(t) = \frac{2}{\sqrt{\beta}} \varphi_t(x)(b_x-b_t) - \frac{2}{\sqrt{\beta}} \int_t^x (b_y-b_t) \varphi_t'(y)dy + \int_t^x (y- \lambda)\varphi_t(y)dy.
	\end{align*}
Applying the two bounds given at the beginning of this proof (\eqref{eq:BMbnd} and \eqref{eq:phibnd}) we obtain that 
	\begin{align*}
	|\varphi_t'(x)-\varphi_t'(t)| \leq C_{x_0,\delta} (\eta_{x_0}+\varphi_t'(t))|x-t|^{1+\delta}
	\end{align*}
	We use the mean value theorem again to get $\varphi_t(x)-(x-t) \varphi_t'(t) = (\varphi_t'(r)- \varphi_t'(t))(x-t)$ which gives us the final bound 
	\begin{align*}
		|\varphi_t(x)  - (x-t) \varphi_t'(t) | \leq C_{x_0,\delta}(\eta_{x_0}+\varphi_t'(t)) |x-t|^{2+\delta}. 
	\end{align*}
\end{proof}

\begin{proposition}
\label{prop:finitediff}
Let $[c,d]$ be any interval, then for any $\gamma<1/2$ and $0<\eps< \frac{\rho}{t-s}$ there exists a constant $C_{\rho, \gamma}$ depending only on $\gamma$ and $\rho$ such that for any $s<t\in [c,d]$ and $i = 1,...,k$ we have 
\begin{equation}
\label{eq:finitediff}
\frac{\eps}{1+\eps}(f'_{i,s}(s))^2 - C_{\rho,\gamma}\eps^{1+\gamma}(t-s)^\gamma \le \frac{\Lambda_{i}(t)- \Lambda_i(s)}{t-s} \le \frac{1+\eps}{\eps}(f'_{i,t}(t))^2 + C_{\rho,\gamma}\eps^{1+\gamma}(t-s)^\gamma.
\end{equation}
\end{proposition}

\begin{remark}
\label{rmk:derivative}
Observe that the choise $\eps = (t-s)^{-\delta}$ meets the conditions of the theorem and converges to $\infty$ as $t\to s$, therefore if we can show that $\lim_{t\to s} f'_{i,t}(t)= f'_{i,s}(s)$ this will be enough to show that the process is differentiable with the derivative at $t$ being given by $(f'_{i,t}(t))^2$. 
\end{remark}

\begin{corollary}
\label{cor:lambdacont}
The process $\G{t}{k}$ is continuous as a function of $t$. 
\end{corollary}

This follows immediately from the inequality in the previous Proposition \ref{prop:finitediff}, simply multiply through by $(t-s)$.

\begin{proof}[Proof of Proposition \ref{prop:finitediff} for $k=1$]
The idea is to use the variational characterization of our eigenvalues to get upper bounds using $\psip{1}{a}$ and $\varp{1}{a}$. In particular we have that
\[
\Lambda_1(t) \le \frac{\fip{ \psip{1}{a}, \Hop{t} \psip{1}{a}}}{\|\psip{1}{a}\|_2^2}, \quad \text{ and } \quad \Lambda_1(s) \le \frac{\fip{ \varp{1}{a}, \Hop{t} \varp{1}{a}}}{\|\psip{1}{a}\|_2^2}.
\]
Before continuing we show that $\|\psip{k}{a}\|_2^2$ is close enough to 1 that it may be neglected for the remainder of the calculations. In particular we have
\begin{align*}
\|\psip{k}{a}\|_2^2 &= \|f_{k,s}\|_2^2 + \int_t^a (x-t)^2\frac{f_{k,s}^2(a)}{(a-t)^2} dx - \int_s^a f_{k,s}^2(x)dx 
\end{align*}
Applying Proposition \ref{prop:almostlinear} we obtain that 
\begin{align}
\label{eq:finitediffnorm}
|\|\psip{k}{a}\|_2^2-1|& \le  \big((a-t)(a-s)^2+(a-s)^3\big)\frac{(f_{k,s}')^2(a)}{3}  +\tilde C_{a,\gamma}(a-s)^{3+ \gamma}.
\end{align}
Taking $a = t+ \eps(t-s)$ we obtain that $|\||\psip{k}{a}\|_2^{-2} -1| \le C(t-s)^3$. These errors may be bounded using the constant $C_{\rho,\gamma}\eps^{1+\gamma}(t-s)^\gamma$ term in equation \eqref{eq:finitediff}. Because of this we will neglect the normalization for the remainder of the argument.

We can then compute the following: 
\begin{align*}
\fip{ \psip{1}{a}, \Hop{t} \psip{1}{a}} - \Lambda_1(s) & = \int_t^a \psip{1}{a}(x)\Hop{t} \psip{1}{a}(x)dx- \int_s^a f_{1,s}(x)\Hop{s} f_{1,s}(x)dx.
\end{align*}
We show that for $a=t+ \eps(t-s)$
\begin{align}
\label{eq:finitediffa}
0\le \fip{ \psip{1}{a}, \Hop{t} \psip{1}{a}} - \Lambda_1(s) \le (f_{1,s}')^2(s)(t-s) \frac{1+\eps}{\eps}+M(1+\eps)^{1+\gamma}(t-s)^{1+\gamma}
\end{align}
In order to do this we must bound $\int_t^a \psip{1}{a}(x)\Hop{t} \psip{1}{a}(x)dx$ above and $\int_s^a f_{1,s}(x)\Hop{s} f_{1,s}(x)dx$ below.  We use Holder continuity of Brownian motion to say that for $c<y<x<a$ and $\gamma<1/2$ fixed with probability 1 there exists a constant $C$ such that 
\[
|b_x-b_y|\le C_{a,\gamma}|x-y|^\gamma.
\]
This gives us 
\begin{align*}
\int_t^a \psip{1}{a}(x)\Hop{t} \psip{1}{a}(x)dx & = \int_t^a \left( \frac{f_{1,s}(a)}{a-t}\right)^2 \left[1+ x(x-t)^2 + (b_x-b_t)(x-t)\right]dx\\
&\le f_{1,s}^2(a)\left( \frac{1}{a-t}+ \frac{(a-t)^2}{4}+t \frac{(a-t)}{3}+ \frac{C_{a,\gamma}(a-t)^{\gamma}}{2+\gamma}\right).
\end{align*}
An application of proposition \ref{prop:almostlinear} allows us to write $f_{1,s}(x) \le (x-s)f_{1,s}'(s)+ C_{a,\gamma}(x-s)^{1+\gamma}$, which leads us to 
\begin{equation}
\label{eq:finitediff1}
\int_t^a \psip{1}{a}(x)\Hop{t} \psip{1}{a}(x)dx \le (f_{1,s}')^2(s)\frac{(a-s)^2}{a-t}+ M_{a,\gamma}(a-t)^{\gamma}(a-s)^2.
\end{equation}
Before continuing with the next bound we make the following observation
\[
\left|\int_s^a (x-s)^{1+\gamma} \Hop s (x-s) ds \right| \le C_{a,\gamma} (a-s)^{\gamma+1},
\]
where $C$ is a random constant depending on the interval $[c,d]$ and the choice of $a$ and $\gamma$. From this we can check that Proposition \ref{prop:almostlinear} implies that
\begin{equation}
\label{eq:finitediff2}
\left| \int_s^a f_{1,s}(x)\Hop{s} f_{1,s}(x)dx-\int_s^a (x-s)f'_{1,s}(s)\Hop{s} (x-s)f'_{1,s}(s)dx\right| \le C_{a,\gamma}(a-s)^{1+\gamma}.
\end{equation}
We finish the lower bound on $\int_s^a f_{1,s}(x)\Hop{s} f_{1,s}(x)dx$ by computing 
\begin{equation}
\label{eq:finitediff3}
\int_s^a (x-s)f'_{1,s}(s)\Hop{s} (x-s)f'_{1,s}(s)dx \ge (f_{1,s}')^2(s)(a-s)- M_{a,\gamma}(a-s)^{2+\gamma}.
\end{equation}
Putting together (\ref{eq:finitediff1}), (\ref{eq:finitediff2}), and (\ref{eq:finitediff3}) we are led to the conclusion that for all $a=t+ \eps(t-s)$ we have 
\begin{equation}
\Lambda_1(t)-\Lambda_1(s) \le (f_{1,s}')^2(s)(t-s) \frac{1+\eps}{\eps}+M\eps^{1+\gamma}(t-s)^{1+\gamma}.
\end{equation}
This leads us one of the inequalities in Propostion \ref{prop:finitediff} for $i =1$. Similar techniques may be used to study $\varp{1}{a}$. These lead to the inequality 
\begin{equation}
\Lambda_1(t)-\Lambda_1(s) \ge  (f_{1,t}')^2(t)(t-s) \frac{\eps}{1+\eps}-M\eps^{1+\gamma}(t-s)^{1+\gamma}.
\end{equation}

\end{proof}

\begin{proof}[Proof of Proposition \ref{prop:finitediff} for $k>1$]
The idea here will be similar to the case where $k=1$, but we now have a more complicated variational characterization to work with which leads to further terms that need to be considered.  We start by introducing the Courant-Fisher characterization of the eigenvalues which is given by 
\begin{equation}
%\label{eq:characterization}
\Lambda_k(t) = \inf_{B\subset L^*_t, \dim B=k} \sup_{g\in B} \frac{\fip{ g, \Hop t g}}{\fip{ g, g }}.
\end{equation}
From this characterization we have 
\[
\Lambda_k(t) \le \sup_{g\in B_s}\frac{\fip{ g, \Hop t g}}{\fip{ g, g }}, \qquad \text{where } \quad B_s = \Span \{\psip{1}{a},...,\psip{k}{a}\}.
\]
We make the following observations: 
For $i \ne j\le k$ 
\[
\fip{ \psip{i}{a},  \psip{j}{a} } = \int_t^a \frac{(x-t)^2}{(a-t)^2}f_{i,s}(a)f_{j,s}(a)dx - \int_s^a f_{i,s}(x)f_{j,s}(x)dx
\] 
And so an application of Proposition \ref{prop:almostlinear} gives us that 
\begin{equation}
\label{eq:finitediff4}
|\fip{ \psip{i}{a},  \psip{j}{a} }| \le \frac{1}{3}(t-s)(a-s)^2f'_{i,s}(s)f'_{j,s}(s) + C(a-s)^{3+\gamma}. 
\end{equation}
Now observe that for all $j$, using bounds identical to those used to prove \eqref{eq:finitediffa}, we can show that 
\begin{equation}
0 \le \fip{ \psip{j}{a}, \Hop{t} \psip{j}{a}} - \Lambda_j(s) \ \le\  (f'_{j,s})^2(s)\frac{(a-s)^2}{a-t}  + M_{a,\gamma}(a-s)^{1+\gamma}.
\end{equation}
Further we get that for $g= c_1 \psip{1}{a} + \cdots + c_k \psip{k}{a} \in B_s$ we have 
\begin{align*}
\fip{ g , \Hop{t} g} &\le \sum_{j=1}^k c_j^2 \left(  \Lambda_j(s) + (f'_{j,s})^2(s)\frac{(a-s)(t-s)}{a-t}  + M_{a,\gamma}(a-s)^{1+\gamma}\right) \\
& \hspace{2cm}+2\Lambda_k(t)  \sum_{j_1<j_2}c_{j_1}c_{j_2}|\fip{ \psip{j_1}{a},\psip{j_2}{a}}|\\
\end{align*} 
Taking $a = t+\eps(t-s)$ and applying the bound in \eqref{eq:finitediff4} we are led to
\begin{align*}
\Lambda_k(t) &\le \sup_{c_1,...,c_n} \|g\|_2^{-2}  \sum_{j=1}^k c_j^2 \left(  \Lambda_j(s) + (f'_{j,s})^2(s)\frac{(1+\eps)(t-s)}{\eps}  + M_{a,\gamma}\big((1+\eps)(t-s)\big)^{1+\gamma}\right) \\
& \hspace{2cm}+2\Lambda_k(t)  \sum_{j_1<j_2}c_{j_1}c_{j_2}(1+\eps)^2(t-s)^3(f'_{j_1,s}(s)f'_{j_2,s}(s) + C) 
\end{align*}
For $t$ sufficiently close to $s$ this is maximal for $c_j =0$ for $j \ne k$ and $c_k = \|\psip{k}{t+\eps(t-s)}\|_2^{-1}$. This is because the $\Lambda_1(s)< \Lambda_2(s)< \cdots <\Lambda_k(s)$ are fixed and distinct with probability 1, but all the remaining terms (except possibly the error term on the first line) converge to $0$ as $t\to s$. The error term is identical in all terms so does not change the optimization. Therefore for some $t$ sufficiently close to $s$ $\Lambda_k(s) + (f'_{k,s})^2(s)\frac{(1+\eps)(t-s)}{\eps}  + M_{a,\gamma}\big((1+\eps)(t-s)\big)^{1+\gamma}$ will be the dominant term and so the right hand side is maximized when all of the $c_i$ are 0 except for $c_k$. By previous argument in line \eqref{eq:finitediffnorm} we have that $c_k^2= 1+ O(t-s)^3$ and so the error we obtain by replacing $c_k$ with $1$ may be neglected.  This gives us that
\[
\Lambda_k(t) - \Lambda_k(s) \le (f'_{k,s})^2(s)\frac{(1+\eps)(t-s)}{\eps}  + M_{a,\gamma}\big((1+\eps)(t-s)\big)^{1+\gamma} ,
\]
which complete the upper bound in the proposition. To complete the lower bound we perform a similar analysis with $B_s= \Span \{\varp{1}{a},...,\varp{k}{a}\}$.
\end{proof}

\begin{lemma}
\label{lemma:unifconv}
The eigenfunctions $f_{1,s},...,f_{k,s}$ of $\Hop t$ converge uniformly on compact subsets to the eigenfunctions $f_{1,t},...,f_{k,t}$ of $\Hop {t}$ as $s \to t$.
\end{lemma}

\begin{proof}
We again reuse the notion and approximating functions introduced in equation (\ref{eq:linearapprox}). We will show that the proposition holds for $s\searrow t$ by using the functions $\varp{k}{a}$. One can show the identical result for $s\nearrow t$ by instead using the functions $\psip{k}{a}$. We consider families of functions of the form
\[
( \varp{1}{t+\eps(t-s)}, \varp{2}{t+\eps(t-s)}, ... ,\varp{k}{t+\eps(t-s)}).
\]
From the proof of \ref{prop:finitediff} we get that $\fip{ \varp{j}{t+\eps(t-s)}, \Hop {s}\varp{j}{t+\eps(t-s)}} \to \Lambda_j(s)$  as $s \searrow t$.

We apply fact 2.2 from \cite{RRV} to get that there exists a subsequence $s_j \searrow t$ and functions $(g_1,...,g_k)$ such that 
\[
( \varphi_{1,s_j,t}^{(t+\eps(t-s_j))}, \varphi_{2,s_j,t}^{(t+\eps(t-s_j))}, ... , \varphi_{k,s_j,t}^{(t+\eps(t-s_j))})\to (g_1,...,g_k)
\]
uniformly on compact subsets in $L^2$ and weakly in $H^1$. It remains to be shown that $(g_1,...,g_k)=(f_{1,t},...,f_{k,t})$ the eigenfunctions of $\Hop{t}$. To complete the picture we use the variational derivative characterization $\frac{d}{d\eps}\fip{ g_j+\eps h, \Hop{t} (g_j+\eps h)}|_{\eps=0} $ to get that $g_j$ satisfies $\Hop{t} g_j = \tilde \Lambda_j g_j$ for some $\tilde \Lambda_j$ and so $g_j$ is an eigenfunction of $\Hop{t}$. The strict ordering of the eigenvalues is enough to complete the picture and give $\tilde \Lambda_j = \Lambda_j(t)$. It follows that $g_j = f_{j,t}$. Therefore we conclude that we in fact have 
\[
( \varphi_{1,s_j,t}^{(t+\eps(t-s_j))}, \varphi_{2,s_j,t}^{(t+\eps(t-s_j))}, ... , \varphi_{k,s_j,t}^{(t+\eps(t-s_j))})\to (f_{1,t},...,f_{k,t})
\]
uniformly on compact subsets in $L^2$ and weakly in $H^1$.

\end{proof}

This weak convergence in $H^1$ suggests that we should have convergence of the derivatives $f_{j,t}'(t)\to f_{j,t_0}'(t_0)$ as $t\to t_0$, and indeed by making use of the fact that the eigenfunctions are almost linear near the boundary point this can be shown. In particular if the eigenfunctions are approximately linear near their endpoint then convergence on compact subsets will imply that the derivatives converge at the end points. 

\begin{lemma}
For all $t_0$, and any $j=1,...,k$ we have $\lim_{t \to t_0} f_{j,t}'(t) = f_{j,t_0}'(t_0)$.
\end{lemma}

\begin{proof}
Let $\eps>0$ We use the following: In a fixed neighborhood of $t_0$ we have the bound from Proposition \ref{prop:almostlinear} with a $C_{\delta, \gamma}$ depending on the neighborhood size $\delta$ and $0<\gamma<1/2$. We now observe that for $x> t\wedge t_0$ in a neighborhood of $t_0$ we have
\begin{align*}
|f_{j,t}'(t)-f_{j,t_0}'(t_0)| & \le |f_{j,t}'(t)-\frac{f_{j,t}(x)}{x-t}|+|f_{j,t_0}'(t)-\frac{f_{j,t_0}(x)}{x-t_0}|+ |\frac{f_{j,t}(x)}{x-t}-\frac{f_{j,t_0}(x)}{x-t_0}|\\
& \le C_{\delta, \gamma} \big((\eta_\delta+f_{j,t}'(t))(x-t)^{1+\gamma}+(\hat\eta_\delta+f_{j,t_0}'(t_0))(x-t_0)^{1+\gamma}\big)\\
& \hspace{9cm}+  |\frac{f_{j,t}(x)}{x-t}-\frac{f_{j,t_0}(x)}{x-t_0}|
\end{align*}
The previous convergence result Lemma \ref{lemma:unifconv} give us that the final term may be made arbitrarily small as $t\to t_0$ for any fixed $x$. Choose $t$ and $t_0$ close enough to $x$ so that the first two terms are bounded by $\eps/3$, then by letting $t$ go to $t_0$ (which does not impact the bounds on the first two terms) we will get that the final term is also bounded by $\eps/3$. Therefore 
\[
\lim_{t\to t_0} f_{j,t}'(t) = f_{j,t_0}'(t_0).
\]
\end{proof}

\begin{proposition}
\label{prop:derivativederivative}
For any fixed $k$ the process $\G k t$ is differentiable as a function of $t$. With the derivatives given by 
\[
\frac{d}{dt} \Lambda_j(t) =( f_{j,t}'(t))^2.
\]
\end{proposition}

See Remark \ref{rmk:derivative} for the proof.

\section{The discrete to continuous convergence} 

In this section we use the machinery developed for the proof of the original soft edge limit in order to show convergence of the $t$ dependent eigenvalue process. To do this we begin by recalling the general convergence theorem from section 5 of \cite{RRV}. 

\begin{theorem}[Theorem 5.1 \cite{RRV}]
\label{thm:conv}
Suppose that $H_n$ is a tridiagonal matrix with 
\begin{align*}
\text{ diagonal}& \qquad 2m_n + m_n y_{n,1}(1),  2m_n + m_n y_{n,1}(2), 2m_n + m_n y_{n,1}(3),...\\
\text{ off-diagonal}& \quad -m_n + \frac{1}{2}m_n y_{n,2}(1),  -m_n + \frac{1}{2}m_n y_{n,2}(2),  -m_n + \frac{1}{2}m_n y_{n,2}(3),...
\end{align*}
and $H= - \partial_x^2 + Y'(x)$ acting on $H'_{\textup{loc}}\mapsto D$ the space of distributions with boundary condition $f(0)=0$ (see \cite{RRV} for further details). Let $Y_{n,i}(x) = \sum_{j=1}^{\lfloor nx\rfloor} y_{n,i}(j)$. For any fixed $k$, the bottom $k$ eigenvalues of $H_n$ converge to the bottom $k$ eigenvalues of $H$ if the following two conditions are met:
\begin{enumerate}
\item (Tightness/Convergence) There exists a process $x\mapsto Y(x)$ such that 
\begin{align*}
(Y_{n,i}(x): x\ge 0 ) & \qquad\  i = 1,2 \quad \text{ are tight in law},\\
(Y_{n,1}(x)+ Y_{n,2}(x): x\ge 0 )&\  \Rightarrow \   (Y(x);x\ge 0) \quad \text{ in law},
\end{align*}
with respect to the Skorokhod topology of paths; see \cite{EthierKurtz} for the definitions.

\item (Growth/Oscillation bound). There is a decomposition
\[
y_{n,i}(k) = \frac{1}{m_n}(\eta_{n,i}(k)+ \omega_{n,i}(k)),
\]
for $\eta_{n,i}(k)\ge 0$, deterministic, unbounded non-decreasing functions $\bar \eta(x)>0, \zeta(x) \ge 1$, and random constants $\kappa_n(\omega)\ge 1$ defined on the same probability space which satisfy the following: The $\kappa_n$ are tight in distribution, and, almost surely,
\begin{align*}
\bar \eta(x)/ \kappa_n - \kappa_n \le \eta_{n,1}(x)+ \eta_{n,2}(x) &\le \kappa_n(1+ \bar \eta(x)),\\
\eta_{n,2}(x) \le 2 m_n^2.\\
|\omega_{n,1}(\xi)- \omega_{n,1}(x)|^2+|\omega_{n,2}(\xi)- \omega_{n,2}(x)|^2 & \le \kappa_n(1+ \bar \eta(x)/ \zeta(x))
\end{align*}
for all $n$ and $x,\xi \in [0, n/m_n]$ with $|x-\xi|\le 1$. 
\end{enumerate}
\end{theorem}

Ram\'irez, Rider, and Vir\'ag show in section 6 of \cite{RRV}, that the tridiagonal model $\Hmat{k}$ defined in \eqref{eq:ktridiagonal} with $k=1$ satisfies the the conditions of the theorem with $m_n=n^{1/3}$ and $Y(x)=\frac{x^2}{2}+ \frac{2}{\sqrt \beta} b_x$. The same arguments may be used to show that for $\Hmat{\lfloor tn^{1/3}\rfloor}$ the same convergence statements hold with $m_n= n^{1/3}$ and $Y(x) = \frac{x^2}{2}- t x + \frac{2}{\sqrt \beta} b_x$. These are two different distributional convergence statements, but with a slight modification of the proof of Theorem \ref{thm:conv} we may show a joint distributional convergence for any finite collection $\{t_1, t_2,...,t_j\}$. 

\begin{proof}[Proof of \ref{thm:discretetocont}]
Let $t_1<\cdots <t_\ell$ be any finite collection of times (possibly negative). We observe using the work in Section 6 of \cite{RRV} that the matrices $\Hmat{\lfloor t_1 n^{1/3}\rfloor},\Hmat{\lfloor t_2 n^{1/3}\rfloor},...,\Hmat{\lfloor t_\ell n^{1/3}\rfloor}$ satisfy the conditions of Theorem \ref{thm:conv} with $m_n= n^{1/3}$ and $Y^{(t_j)}(x) = \frac{x^2}{2}- t_j x + \frac{2}{\sqrt \beta} b_x$. Moreover we have 
\[
y_{n,i}^{(t_j)}(k) = y_{n,i}^{(t_1)}(k+\lfloor t_j n^{1/3}\rfloor).
\]
Because of this identity the if the conditions of Theorem \ref{thm:conv} hold for $t_1$ then they also hold for $t_2,...,t_\ell$.
Therefore for any subsequence we can extract a further subsequence such that we have the following joint distributional convergence:
\begin{align*}
( \int_0^x \eta_{n,i}^{(t_j)}(y)dy; x\ge 0 ) & \Rightarrow (\int_0^x \eta_i^{(t_j)} (y)dy; x\ge 0), \\
(Y_{n,i}(x); x\ge 0 ) & \Rightarrow ( \frac{x^2}{2}- t_j x +\frac{2}{\sqrt \beta} b_i(x+t_j) ; x\ge 0), \qquad j = 1,...,\ell \\
\kappa_n^{(t_j)} & \Rightarrow \kappa^{(t_j)}
\end{align*}
where the first line converges uniformly on compact subsets and the second in the Skorokhod topology. Notice that the brownian motions $b_i$ that appear are the same for all $j$. The Skorokhod representation theorem (see Theorem 1.8, Chapter 3, or \cite{EthierKurtz}) gives us that there exists a probability space so that the necessary convergence statements hold with probability 1. This allows us to reduce to working with the deterministic case and the remainder of the proof goes through unchanged. In all at this point we have proved that 
\[
\left\{ \lambda_1(\lfloor t_j n^{1/3}\rfloor),  \lambda_2(\lfloor t_j n^{1/3}\rfloor),..., \lambda_k(\lfloor t_j n^{1/3}\rfloor) \right\}_{j=1,...,\ell} \Rightarrow \left\{ \Lambda_1(t_j ),  \Lambda_2(t_j),..., \lambda_k(t_j) \right\}_{j=1,...,\ell}
\]
where $\Lambda_1(t_j)<  \Lambda_2(t_j)< \cdots $ are the eigenvalues of the operator
\[
H^{(t_j)} = - \frac{d^2}{dx^2} + x- t_j + \frac{2}{\sqrt \beta}db(x+t_j)
\]
acting on functions in $L^*[0,\infty)$. These are exactly the eigenvalues of the operator $\Hop{t_j}$ defined in \eqref{eq:Ht}. Therefore we have convergence of finite dimensional distributions which completes the proof of Theorem \ref{thm:discretetocont}.
\end{proof}

\section{Distribution of the derivatives}

We need to begin by showing that the eigenvalues of the discrete operator follow an approximately linear pattern where the `slope' is determined by the first entry of the eigenvector. Because we know the distribution of the spectral weights which are found in these first entries we can then use this property to determine the distribution of the eigenfunctions in the limit. This will in turn give us the derivative of the process as desired. 

Before continuing on to the proof of the proposition we will need some information on the distribution of $v_1$. 

\begin{lemma}[Dumitriu-Edelman \cite{DE}]
\label{lemma:dirichlet}
The squares of the spectral weights $q_i = (v_1^{(i)})^2$ associated to the tridiagonal model in \eqref{eq:tridiagonal} are Dirichlet with parameters $(\frac{\beta}{2},...,\frac{\beta}{2})$. These weights are the square of the first entry of each normalized eigenvector. The marginal distribution of a single spectral weight is 
\[
q_i \sim \textup{Beta} \Big(\frac{\beta}{2},(n-1)\frac{\beta}{2}\Big), \qquad Eq_i = \frac{1}{n}, \qquad \var q_i = \frac{\beta (n-1)}{n^2(\beta n+2)}.
\]
\end{lemma}

\begin{lemma}
\label{lem:gammalimit}
Let $q_i$ be as above, then for any $\ell \in \mathbb{N}$ we have
\[
n(q_1, q_2,...,q_\ell ) \Rightarrow (\Gamma_{1}, \Gamma_{2},..., \Gamma_{\ell}), \qquad \Gamma_{i}\sim \Gamma(\tfrac{\beta}{2}, \tfrac{2}{\beta})
\]
where $\Gamma_{1},...,\Gamma_{\ell}$ are independent. This is using the shape and scale convention for Gamma random variables.
\end{lemma}

\begin{proof}

From Lemma \ref{lemma:dirichlet} we know that the spectral weights $q_1,...,q_n$ are exchangeable with distribution Dirichlet$(\frac{\beta}{2}, ...,\frac{\beta}{2})$ and $(v^{(i)}_1)^2 = q_i$. We now use the following characterization of a Dirichlet distribution: Let $X_1, X_2,...,X_n$ be independent identically distributed with $X_i \sim $Gamma $(\frac{\beta}{2},1)$, then for 
\[
q_i = \frac{X_i}{X_1+ X_2 + \cdots +X_n}, \quad \text{ we get } \quad (q_1, q_2, ...,q_n) \sim \textup{Dirichlet}( \tfrac{\beta}{2},..., \tfrac{\beta}{2}).
\]
We observe that by the strong law of large numbers $(X_1+ \cdots + X_n)/n \to \frac{\beta}{2}$ in probability, and note that for $\eta \sim $ Gamma $(k,\theta)$, $c \eta \sim $ Gamma $(k,c\theta)$. Therefore this characterization this is enough to give the desired joint convergence statement.
\end{proof}

We now prove a proposition that will show that the eigenvector is close enough to linear that in the limit we will get that the derivative at 0 is determined by the distribution of the spectral weights. 
\begin{proposition}
\label{lem:discretelinear}
Let $v$ be the eigenvector associated to $\lambda_i(n,0)$ the $i$th lowest eigenvalue of $\Hmat{0}$ defined in \eqref{eq:ktridiagonal}. For any $\eps>0$ there exists a set $\mathcal{A}_\eps \subset \Omega$ with $P(\mathcal{A}_\eps)>1-\eps$,  and $x_0>0$ sufficiently small such that for all $t \in [0,x_0]$ and $\omega \in \mathcal(A)_\eps$
\[
|v_{\lfloor tn^{1/3}\rfloor}- \lfloor tn^{1/3}\rfloor v_1| \le C \frac{t^2}{n^{1/6}} \sqrt{\frac{x_0}{\eps}}.
\]
\end{proposition}

\begin{proof}
Recall that we're working with the matrix $\Hmat{0} = n^{1/6}(2\sqrt{n}- A_\beta)$. To start let's scale the $n^{2/3}$ out of the leading term then the resulting matrix has the form
\[
n^{-2/3}H_n^\beta = \left[ \begin{array}{cccc} 2 + \rho_1 & -1 + r_1 &&\\
-1+ r_1 & 2+ \rho_2 & -1+ r_2 & \\
& \ddots & \ddots & \ddots \end{array}\right].
\]
Under this we have that $\rho_i \sim \frac{1}{\sqrt n} \mathcal{N}(0,2)$ and $r_{i} = \frac{i}{2 n} + \frac{1}{\sqrt n} \eta_i$ where $\eta_i$ is an order 1 random mean 0 variable with Gaussian tails. Notice that this rescaling does not change the distribution of the eigenvectors, but the lowest eigenvalues will now be on the order of $n^{-2/3}$. Before we start we give two bounds: 
For the first we use Doob's martingale inequality to get that
\begin{equation}
\label{eq:RVsumbnd}
P\left( \max_{1\le k \le x_0n^{1/3}}  \frac{1}{\sqrt{n}}\sum_{\ell=1}^k \eta_\ell  \ge \frac{\sqrt{x_0}}{n^{1/3}} M \right) \le \frac{C}{M^2},
\end{equation}
With a similar bound holding for the $\rho$'s.
%First suppose we have a sequence $X_1, X_2,...$ of independent random variables with $EX_i=0$ and $EX_i^2= \sigma^2$ and let $\gamma>1/2$, then $\frac{1}{k^{\gamma}}\sum_{\ell=1}^k X_\ell\to 0$ a.s.. This give us that on a set of full measure there exists a random constant $C_\gamma$ such that 
%\begin{equation}
%\label{eq:RVsumbnd}
%\sum_{j=1}^k \eta_i \le C_\gamma \frac{k^{\gamma}}{\sqrt{n}} \qquad \text{ for all } k = 1, 2, 3,...
%\end{equation}
Next, suppose that $a_1, a_2, a_3,...$ is a sequence such that $|(a_{k+1}-a_k) - a_1|\le \tau$ for all $k\ge 1$, then 
\begin{equation}
\label{eq:incbnd}
|a_k| \le k(a_1+ \tau).
\end{equation}

Now we move on the the main part of the proof. Suppose that $v$ solves $\Hmat{0} v = \lambda_j v$ with $\|v\|=1$. We can check that $v_k$ satisfies the following: 
\begin{align}
v_{k+1}-v_k &= (v_k-v_{k-1})+ r_{k-1}v_{k-1}+ (\rho_k-\lambda)v_k+ r_kv_{k+1}\notag \\
&=  \sum_{\ell=1}^{k} r_{\ell-1}v_{\ell-1}+ \rho_\ell v_\ell -\lambda v_\ell+ r_\ell v_{\ell+1} \label{eq:aaa}
\end{align}
if we assume that $|v_{k+1}-v_k- v_1| \le 1/\sqrt{n}$ for $k \le x_0 n^{1/3}$ then using \eqref{eq:incbnd} we get that $|v_k| \le k(v_1+ \eta/\sqrt{n})$. We rewrite the final term in the sum:
\[
\sum_{\ell=1}^k r_\ell v_{\ell+1} = \sum_{\ell=1}^k (R_{\ell+1}-R_{\ell}) v_{\ell+1} = v_{k+1} R_{k+1} + \sum_{\ell=1}^k R_\ell(v_{\ell}-v_{\ell-1})
\]
By \eqref{eq:RVsumbnd} there exists a set $\mathcal{A}_M$ of size $1-\frac{C}{M^2}$ such that $|R_k| \le  \frac{\sqrt{x_0}M}{n^{1/3}} $ for some fixed $C$. On this set we get 
\[
\sum_{\ell=1}^k r_\ell v_{\ell+1} \le 2(v_1+1 /\sqrt{n}) \frac{\sqrt{x_0}M }{n^{1/3}} \cdot k.
\]
This same argument holds for the other two random terms in \eqref{eq:aaa}. Using Lemma \ref{lemma:dirichlet} and the observation $v_1= q_j$ (which is on the order of $1/n$) and the fact that $\lambda_j$ is order $1/n^{2/3}$ we get that on a set of size $1-3\frac{C}{M^2}$.
\[
|v_{k+1}-v_k - v_1| \le |v_{k+1}-v_k|-|v_1| \le C \frac{1}{\sqrt{n}} \frac{\sqrt{x_0}M k}{n^{1/3}}.
\]
This validates our original assumption that $|v_{k+1}-v_k- v_1| \le 1/\sqrt{n}$ for $k \le x_0 n^{1/3}$ where $x_0$ is sufficiently small. 
%For $k = tn^{1/3}$ this yields the bounds 
%\[
%|v_{\lfloor tn^{1/3}\rfloor+1}-v_{\lfloor tn^{1/3}\rfloor} - v_1|   \le C \frac{\sqrt{x_0} M t}{\sqrt{n}}
%\]
%Take $M = \frac{C}{\sqrt \eps}$ for some constant $C$ to complete the proof.
%
Finally this yields 
\[
|v_{k+1}-(k+1)v_1| \le C \frac{1}{\sqrt{n}} \frac{\sqrt{x_0}M k(k+1)}{2n^{1/3}},
\]
which for  $k = \lfloor tn^{1/3}\rfloor$ give us 
\[
|v_{\lfloor tn^{1/3}\rfloor}- \lfloor tn^{1/3}\rfloor v_1| \le \tilde C \frac{1}{n^{1/6}}t^2 \sqrt{x_0}M \quad \text{ for some constant } \tilde C.
\]
Take $M = \frac{C'}{\sqrt \eps}$ for some constant $C'$ to complete the proof.

\end{proof}

\begin{proposition}
At any fixed time $t>0$ the derivatives of the process $\G{t}{k}$ in $t$ are independent with distribution
\[
\frac{d}{dt} \Lambda_j(t) = \Gamma_j(t),\quad \text{ for i.i.d.}\quad  \Gamma_j(t) \sim \textup{ Gamma }(\frac{\beta}{2},\frac{2}{\beta}),
\] 
for $j = 1, 2, ...,k$.
\end{proposition}

\begin{proof}
We begin with the observation that the process $\G{t}{k}-t$ stationary in $t$ and therefore the distribution of the derivative for all $t$ is determined by the distribution of the derivate at $t=0$. Let $v^{(1)},...,v^{(\ell)}$ be the eigenvectors of $\Hmat{0}$. We will show that $\{(f_{i,0}'(0))^2\}_{i=1}^\ell$ are independent with the desired distribution by showing that $\lim_{n\to \infty} \sqrt{n} v^{(i)}_{1} \ed f_{i,0}'(0)$, which together with Lemma \ref{lem:gammalimit} will imply the result.

We embed $v^{(1)},...,v^{(\ell)}$ the eigenvectors of $\Hmat{0}$ as step functions with $v^{(i,n)}(x)=n^{1/6} v^{(i)}_{\lfloor x n^{1/3}\rfloor}$ in $L^2[0,n^{2/3})$, and we can check that $\|v^{(i,n)}\|_{L^2}= (n^{1/6})^2 n^{-1/3} \|v^{(i)}\|_2=1$. Similarly we embed the vector $L_{v^{(i)}} = [v_1^{(i)}, 2 v_1^{(i)} , 3v_1^{(i)} , ..., \lfloor tn^{1/3} v_1^{(i)} \rfloor, 0,...,0 ]^t$ as the step function $L_{v^{(i)}}^{(n)}(x) = n^{1/6} v_1^{(i)}\lfloor xn^{1/3}\rfloor$ for $x<t$. Here we perform the truncation so that the $L_2$ norm remains bounded.  From the proof of Theorem \ref{thm:discretetocont} (see Lemma 5.8 \cite{RRV}) we get that there exists a subsequence along which 
\begin{align}
\label{eq:efuncconvergence}
\begin{array}{c}
(v^{(1,n)}, v^{(2,n)},\ldots , v^{(\ell,n)})\\
(L_{v^{(1)}}^{(n)}(x), L_{v^{(2)}}^{(n)}(x), \ldots , L_{v^{(\ell)}}^{(n)}(x) )\\
(v^{(1,n)}(x)- L_{v^{(1)}}^{(n)}(x), \ldots  , v^{(\ell,n)}(x)-L_{v^{(\ell)}}^{(n)}(x))
\end{array} \Rightarrow 
\begin{array}{c}
(f_{1,0}, f_{2,0},\ldots , f_{\ell,0})\\
(x\sqrt{\Gamma_{1}},x \sqrt{ \Gamma_{2}},\ldots , x\sqrt{\Gamma_{\ell}})\\
(f_{1,0}(x)-x\sqrt{\Gamma_{1}}, \ldots ,f_{\ell,0}(x)-x\sqrt{\Gamma_{\ell}})
\end{array}
\end{align}
jointly in law as functions of $x$, where the $\Gamma_i$ are defined as in Lemma \ref{lem:gammalimit}. 
%Similarly we can embed the vector $L_{v^{(i)}} = [v_1^{(i)}, 2 v_1^{(i)} , 3v_1^{(i)} , ...]^t$ as the step function $L_{v^{(i)}}^{(n)}(x) = n^{1/6} v_1^{(i)}\lfloor xn^{1/3}\rfloor$ and using the same machinery get the convergence
%\begin{equation}
%\label{eq:linearfuncconvergence}
%(L_{v^{(1)}}^{(n)}(x), L_{v^{(2)}}^{(n)}(x), ... , L_{v^{(\ell)}}^{(n)}(x) ) \Rightarrow (x\sqrt{\Gamma_{1}},x \sqrt{ \Gamma_{2}},..., x\sqrt{\Gamma_{\ell}}) \quad \text{ as functions in } x,
%\end{equation}
%where the $\Gamma_i$ are defined as in Lemma \ref{lem:gammalimit}.

From Proposition \ref{lem:discretelinear} we have that for any $\eps>0$ there exist $\mathcal{A}_\eps\subset \Omega$ with $P(\mathcal{A}_\eps)>1-\eps$ such that for all $x<x_0$
\[
|n^{1/6}v^{(i)}_{\lfloor xn^{1/3}\rfloor}- n^{1/6} \lfloor xn^{1/3}\rfloor v^{(i)}_1| \le C x^2 \sqrt{\frac{x_0}{\eps}}.
\]
From the distributional limit it follows that the same hold for $f_{i,0}(x)-x\sqrt{\Gamma_{i}}$. On the set $\mathcal{A}_{\eps,\infty}$ we get that 
\[
\lim_{x\to 0} \frac{f_{i,0}(x)- x\sqrt{\Gamma_{i}}}{t} = 0
\]
which is equivalent to $f_{i,0}'(0)= \sqrt{\Gamma_{i}}$. Since $\eps$ may be made arbitrarily small this is enough to give us that it holds with probability 1.
%
%\[
%n^{1/2}(v_1^{(1)},...,v_1^{(\ell)}) \Rightarrow (f_{1,0}'(0), f_{2,0}'(0),..., f_{\ell,0}'(0)).
%\]
%The continuous mapping theorem immediately gives us convergence of the squared entries to the squared derivatives. Therefore if we can compute the limiting distribution of the first entry of each of the first $\ell$ eigenvectors of the matrix $\Hmat{0}$ this will determine the distribution of $\{(f_{i,0}'(0))^2\}_{i=1}^\ell$, which by Proposition \ref{prop:derivativederivative} gives us the distribution of the derivative of the eigenvalue process in $t$.
 \end{proof}

%%%%%%%%%%%%%%%%%%%%%%%%%%%%%%%%%%%%%%%%%%%%%%%%%%%%%%%%%%%%%%%%%%%%%%%%%%%%%%%%%%%%%%%%%%%%%%%%%%%%%%%%%%%%%%%%%%%%%%%%%%%%%%%%%%%%%%%%%%%%%%%%%%%%%%%

\bibliographystyle{amsplain}
\bibliography{rmt.bib}

\end{document}